\documentclass[12pt]{amsart}
\usepackage{amssymb}
\title{Highest weight $\sl_2$-categorifications I: crystals}
\author{Ivan Losev}
\address{Department
of Mathematics, Northeastern University, Boston MA 02115 USA}
\email{i.loseu@neu.edu}
\thanks{Supported by the NSF grant DMS-0900907}
\renewcommand{\sl}{\mathfrak{sl}}
\newcommand{\Cat}{\mathcal{C}}
\newcommand{\OCat}{\mathcal{O}}
\newcommand{\Hecke}{\mathcal{H}}
\newcommand{\C}{\mathbb{K}}
\newcommand{\Q}{\mathbb{Q}}
\newcommand{\Z}{\mathbb{Z}}
\newcommand{\wt}{\operatorname{wt}}
\newcommand{\End}{\operatorname{End}}
\newcommand{\g}{\mathfrak{g}}
\newcommand{\cont}{\operatorname{cont}}
\newcommand{\Hom}{\operatorname{Hom}}
\newcommand{\Res}{\operatorname{Res}}
\newcommand{\Ind}{\operatorname{Ind}}
\newcommand{\Ext}{\operatorname{Ext}}
\newcommand{\ind}{\operatorname{ind}}
\newcommand{\head}{\mathsf{head}}
\newcommand{\res}{\operatorname{res}}
\newcommand{\gl}{\mathfrak{gl}}
\newcommand{\p}{\mathfrak{p}}
\newcommand{\lf}{\mathfrak{l}}
\newcommand{\Fi}{\operatorname{F}}
\newcommand{\h}{\mathfrak{h}}
\newcommand{\GL}{\operatorname{GL}}
\newcommand{\Span}{\operatorname{Span}}

\unitlength=1mm   \numberwithin{equation}{section}
\newtheorem{Thm}{Theorem}[section]
\newtheorem{Prop}[Thm]{Proposition}

\newtheorem{Lem}[Thm]{Lemma}
\theoremstyle{definition}
\newtheorem{Ex}[Thm]{Example}

\newtheorem{Rem}[Thm]{Remark}

\numberwithin{table}{section} \oddsidemargin=0cm
\begin{document}
\begin{abstract}
We define highest weight categorical actions of $\sl_2$ on  highest weight categories
and show that basically all known examples of categorical $\sl_2$-actions on highest weight categories (including
 rational and polynomial representations of  general linear groups, parabolic categories $\mathcal{O}$
of type $A$, categories $\mathcal{O}$ for cyclotomic Rational Cherednik algebras) are highest weight
in our sense. Our main result is an explicit combinatorial description of (the labels of) the crystal
on the set of simple objects. A new application of this is to determining the supports of simple modules over
the cyclotomic Rational Cherednik algebras starting from their labels.
\end{abstract}
\maketitle
\section{Introduction}
Categorical actions of Kac-Moody algebras were introduced by Chuang and Rouquier, \cite{CR}, in the case of
$\sl_2$ and by Rouquier, \cite{Rouquier_2Kac}, in general. These actions proved to be very useful in Representation
Theory. For example, in \cite{CR} they were used to prove the Broue abelian defect conjecture. It is worth
mentioning that related techniques were used in several papers before \cite{CR} although they were not formalized.

Roughly speaking, a categorical $\sl_2$-action on an abelian category $\Cat$ is a pair of biadjoint functors $E,F$
 together with certain natural transformations. A categorical action of an arbitrary Kac-Moody
algebra includes functors $E_i,F_i$ corresponding to Cartan generators and can be regarded as a collection
of categorical $\sl_2$-actions subject to some compatibility conditions.
Examples of the categories $\Cat$ that can be equipped with
categorical actions of a Kac-Moody Lie algebra $\g$ (below such categories are called $\g$-categorifications)
include many categories of interest for Representation theory, see, e.g., \cite[Section 7]{CR}. For instance, one can consider
the sum $\bigoplus_{n\geqslant 0} \C S_n-\operatorname{mod}$ of the categories of  all finite dimensional $\C S_n$-modules,
where $\C$ is an algebraically closed field. This category comes equipped with a categorical action of
$\g$, where $\g=\hat{\sl}_p$ if   $\C$ is a field of characteristic $p>0$, and $\g=\gl_\infty$ if the characteristic is $0$.
This categorification comes from the induction and restriction functors. There is a similar in spirit ``higher level'' construction
for cyclotomic Hecke algebras.

Another class of examples considered in \cite{CR} comes from the representation theory of algebraic groups
or Lie algebras. For example, we can consider the category $\GL_n(\C)-\operatorname{mod}$ of rational representations
of $\GL_n(\C)$. It comes with a categorical $\g$-action (with $\g$ as above) that is induced from tensoring with $\C^n$
and $(\C^n)^*$. This category has a polynomial analog $\bigoplus_{d\geqslant 0} \operatorname{Rep}^d(\GL)$, where
$\operatorname{Rep}^d(\GL)$ stands for the ``stable'' category of polynomial representations of $\GL$ of degree $d$
(``stable'' means that we consider the representations of $\GL_n$ with $n\geqslant d$).
Also there are higher level analogs of these categories: parabolic categories $\mathcal{O}$ over $\mathfrak{gl}_n$,
where one has categorical actions of $\gl_\infty$. Yet another, more recent, example
comes from the representation theory of cyclotomic Rational Cherednik algebras, \cite{Shan},\cite{GM}.

The categories described in the previous paragraph all have an additional structure, a highest weight structure
(another name: a quasi-hereditary structure). That is, they have a distinguished collection of objects, standard
objects, that have properties of Verma modules in the BGG category $\mathcal{O}$. Two natural questions then arise.
First, what are reasonable compatibility relations between  highest weight and  categorification structures?
Second, assuming that the structures are compatible, what implications for the representation theory does this have?
In this paper we give some version of an answer to the first question (in the case of $\sl_2$) and also
describe an application: a combinatorial description of the crystal associated to a categorical action.

The crystal under consideration is on the set of the simples in $\Cat$. The crystal operators will be recalled
below. In all of highest weight categories recalled above the simples are parameterized by some combinatorial
objects. For instance, the simples in $\GL_n(\C)-\operatorname{mod}$ are parameterized by (dominant) weights,
while the simples in  $\bigoplus_{d\geqslant 0} \operatorname{Rep}^d(\GL)$ are parameterized by partitions.
In the case of the latter category, the crystal has an explicit representation theoretic meaning: it describes the $d-1$ degree part of
the  socle (i.e., the sum of all simple subobjects) in the restriction of an irreducible object in $\operatorname{Rep}^d(\GL_n(\C))$
to $\GL_{n-1}(\C)$, see, for example, \cite[Theorem C]{BK_restr}.

The combinatorial description of the crystal was known previously for all the categories above, see \cite{Kleshchev}
(for $\operatorname{Rep}(\GL_n)$), \cite{BKshift} (for the parabolic categories $\mathcal{O}$),
with an exception of the categories $\mathcal{O}$ of cyclotomic Rational Cherednik algebras, there the description was only known under
restrictions on parameters participating in the definition of the algebra, \cite{GL}.
For example, in the case of polynomial representations of $\GL$ the description
is given in terms of addable and removable boxes in Young diagrams and we will see that
this is a more or less general pattern.  We remark that the descriptions of \cite{Kleshchev},\cite{BKshift} require some
non-trivial and technical computations. Let us also remark that in the Cherednik case the crystal,
perhaps, has the most transparent representation theoretic meaning: it carries some information about the
supports of irreducible modules.

The main goal of this paper is to produce a combinatorial description of
the crystal of a highest weight $\sl_2$-categorification in a uniform way. The crucial part of the argument will
be to verify that certain Ext's between standard objects and irreducible objects vanish. We remark that our problem is
not to determine the crystal up to an isomorphism (like in, say, \cite{LV}) -- this is known and relatively easy in all
examples we consider.  The point is that for each label (e.g., Young diagram) we can completely describe in combinatorial terms how the crystal operators act on it.

Let us mention another source of  examples of highest weight categorifications, where the description of a crystal
is known. This is a construction due to Webster, \cite{Webster}. In a way, the structure of a crystal is an easy corollary
of Webster's (highly nontrivial) construction. We remark, however, that to identify Webster's categories with
classical ones is also a nontrivial task, see, for example, \cite{SW}.

Let us now describe the structure of this paper. Sections \ref{S_sl_2_cat},\ref{S_hw}
do not contain any new material.  In Section \ref{S_sl_2_cat} we will
recall some standard facts about $\sl_2$-categorifications including the crystal
structure on the set of simple objects. In Section \ref{S_hw} we will recall the definition
of a highest weight category and examples of highest weight categories mentioned above
together with categorical actions. Then in Section \ref{S_hw_sl2_cat} we will
define highest weight categorical $\sl_2$-actions and explain  why the actions
recalled in Section \ref{S_hw} are highest weight. Finally, in Section \ref{S_crystal}
we will state and prove our main result, Theorem \ref{Thm:crystal_main}, on the combinatorial description of the crystal.
Then we will recall a relationship between the crystal and the supports of irreducible
modules for the cyclotomic Rational Cherednik algebras obtained in \cite{Shan},\cite{shanvasserot}.

{\bf Acknowledgements.} I would like to thank Jonathan Brundan, Iain Gordon, Alexander Kleshchev, Peter Tingley
and Ben Webster for stimulating discussions.

\section{$\sl_2$-categorifications}\label{S_sl_2_cat}
\subsection{$\sl_2$-categorifications: definitions}\label{SS_sl2_def}
Our exposition here follows \cite{CR}, where $\sl_2$-\!\! categorifications
(=categorical $\sl_2$-actions) were introduced. Below when we consider the
Lie algebra $\sl_2$ we always mean an algebra over $\Q$.

Let $\C$ be a field.
Let $\Cat$ be an artinian $\C$-linear abelian category. Following \cite[5.1]{CR}, by a {\it weak} $\sl_2$-categorification on
$\Cat$ one means a pair of exact endofunctors $E,F$ of $\Cat$, where $E$ is left adjoint to $F$
with fixed unit and counit morphisms $\epsilon:\operatorname{Id}\rightarrow FE, \eta:EF\rightarrow \operatorname{Id}$. These data are supposed  to satisfy
\begin{itemize}
\item The action of the operators $e:=[E],f:=[F]$ on $[\Cat]:=\Q\otimes_{\Z} K(\Cat)$ induced by the functors
$E,F$ respectively produce a locally finite action of $\sl_2$.
\item The classes of simples in $[\Cat]$ are weight vectors.
\item $F$ is isomorphic to the left adjoint of $E$.
\end{itemize}


A weak $\sl_2$-categorification on $\Cat$ is called an $\sl_2$-categorification, \cite[5.2]{CR}, if it comes equipped,
in addition, with functor morphisms $X\in \End(E), T\in \End(E^2)$ and numbers $q\in \C^\times$
and $a\in \C$ with $a\neq 0$ provided $q\neq 1$ subject to the following conditions:
\begin{itemize}
\item $(1_ET)\circ (T1_E)\circ (1_E T)=(T1_E)\circ (1_ET)\circ (T1_E)$ in $\End(E^3)$. Here and below
$1_ET$ denotes the endomorphism of $E^3$ that is obtained by applying $T$ to the second and third
copies of $E$. The notation $T1_E$ has a similar meaning.
\item $(T+1_{E^2})\circ(T-q 1_{E^2})=0$.
\item $T\circ (1_EX)\circ T=\begin{cases}q X1_E, &q\neq 1,\\
X1_E-T, &q=1.\end{cases}$
\item $X-a$ is nilpotent.
\end{itemize}

The notion of a morphism of  (weak) categorifications is introduced in a natural way,
see  \cite[5.1,5.2]{CR}.

Below we sometimes, following \cite{CR}, write $E_+$ for $E$ and $E_-$ for $F$.

To finish the subsection let us provide a prototypical example of a categorification, see
\cite[7.2]{CR}.  Consider the affine Hecke algebra $\Hecke_q^{aff}(n)$, where $q\neq 1$, generated by
elements $T_1,\ldots,T_{n-1},X_1,\ldots,X_n$ subject to the usual relations.
Fix a collection $\underline{Q}$ of nonzero complex numbers $Q_0,\ldots,Q_{\ell-1}$ and consider
the cyclotomic quotient $\Hecke^{\underline{Q}}_q(n)$ of $\Hecke_q^{aff}(n)$
by the relation $\prod_{i=0}^{\ell-1} (X_1-Q_i)=0$. Consider the category $\Cat:=\bigoplus_{n\geqslant 0}\Hecke^{\underline{Q}}_q(n)$-$\operatorname{mod}$. For every $n$ we have a natural inclusion $\Hecke^{\underline{Q}}_q(n-1)\hookrightarrow \Hecke^{\underline{Q}}_q(n)$
making $\Hecke^{\underline{Q}}_q(n)$ into a free (left or right) $\Hecke^{\underline{Q}}_q(n-1)$-module.
So the category $\Cat$ comes equipped with exact restriction (from $n$ to $n-1$) and induction (from $n-1$ to $n$)
endofunctors. Clearly, $X_n\in \Hecke^{\underline{Q}}_q(n)$ commutes with $\Hecke^{\underline{Q}}_{q}(n-1)$.
Pick $a\in \C^\times$ and consider the functors $E=\bigoplus_n E_n,F=\bigoplus F_n$.  Here
$E_n$ is the generalized eigen-functor for the action of $X_n$ on $\Res: \Hecke^{\underline{Q}}_q(n)$-$\operatorname{mod}\rightarrow
\Hecke^{\underline{Q}}_q(n-1)$-$\operatorname{mod}$ with eigenvalue $a$.  The functor $F_n$ is obtained in a similar
way from $\Ind: \Hecke^{\underline{Q}}_q(n-1)$-$\operatorname{mod}\rightarrow \Hecke^{\underline{Q}}_q(n)$-$\operatorname{mod}$.



\subsection{$\sl_2$-categorifications: properties}\label{SS_SL2_prop}
Here we will list some properties of $\sl_2$-\!\! categorifications.

\begin{Prop}[\cite{CR}, Proposition 5.5]\label{Lem:cat_wt_decomp}
Let $\mathcal{C}$ be a weak $\sl_2$-categorification. Let $\Cat_a$ denote the full
subcategory of $\Cat$ consisting of all objects whose class in $[\Cat]$ lies in the $a$-weight
space. Then $\Cat=\bigoplus_{a\in \Z}\Cat_a$ and $E\Cat_a\subset \Cat_{a+2},
F\Cat_a\subset \Cat_{a-2}$.
\end{Prop}

Following \cite{CR}, we introduce some notation. For a simple object $S$ in a weak  $\sl_2$-categorification
$\Cat$ set $h_?(S):=\max\{i| E_?^i S\neq 0\}$ with $?=+,-$ and $d(S)=h_-(S)+h_+(S)+1$. It is clear, in particular,
that for $S\in \Cat_a$ we have $a=h_-(S)-h_+(S)$.


\begin{Lem}[\cite{CR}, Lemma 5.11]\label{Lem:obj_subs}
Let $\Cat$ be a weak $\sl_2$-categorification. Let $M$ be an object in $\Cat$ such that  for any its simple
quotient $S$ one has $d(S)\geqslant d$. Then for any simple quotient $T$ of $E^i_\pm M$ we have $d(T)\geqslant d$.
The same is true for subobjects instead of quotients.
\end{Lem}


The following theorem summarizes some results obtained in \cite[Lemma 5.13, Proposition 5.20]{CR}
(a part of (3) is actually in the proof of Proposition 5.20).

\begin{Prop}\label{Prop:simples}
Let $\Cat$ be an $\sl_2$-categorification, $i\leqslant n$ be non-negative integers,
and $S$ a simple in $\Cat$ with $h_+(S)=n$. Then the following holds.
\begin{enumerate}
\item The functor $E^i$ decomposes into the sum of $i!$ copies of a functor $E^{(i)}$.
\item  The socle and the head of $E^{(i)}S$ are isomorphic to the same simple object,
say $T$ (depending on $i$).
\item Furthermore,
for any other simple subquotient $T'$ of $E^{(i)}S$ we have $E^{n-i}T'=0$.
\end{enumerate}
The same holds if we replace $E$ with $F$ and $h_+(\bullet)$ with $h_-(\bullet)$.
\end{Prop}

Another useful and interesting property of an $\sl_2$-categorification $\Cat$
that will not be used in this paper is that on $\Cat_a$ one has
$EF\oplus \operatorname{Id}^{\oplus \max(0,-a)}\cong FE\oplus \operatorname{Id}^{\oplus \max(a,0)}$.

\subsection{More general categorifications}
To a simply laced quiver $Q$  one can assign the Kac-Moody algebra $\g(Q)$.
Let $I$ be the set of vertices of $Q$ and let $e_i,f_i$ be the Cartan generators
of $\g(Q)$. Then one can introduce the notion of a weak $\g(Q)$-categorification
on $\Cat$  similarly to the above (``locally finite'' becomes ``integrable''). This structure includes
exact functors $E_i,F_i, i\in I,$ together with fixed adjointness morphisms.
The notion of a genuine $\g(Q)$-categorification is more complicated than in
 the $\sl_2$-case, see \cite{Rouquier_2Kac} (with an exception of finite
and affine type $A$). We will not need the definition of a $\g(Q)$-categorification. The only
thing that we will use is that if $E_i,F_i, i\in I,$ define a $\g(Q)$-categorification,
then for each $i$ the pair $E_i,F_i$ defines an $\sl_2$-categorification.

\subsection{Crystals}\label{SS_crystals}
In this paper we consider $\sl_2$-crystals. An $\sl_2$-crystal  is a set $C$ equipped
with maps $\wt:C\rightarrow \Z, \tilde{e},\tilde{f}:C\rightarrow C\sqcup\{0\},
h_+,h_-:C\rightarrow \Z_{\geqslant 0}$. These maps should satisfy the following conditions
\begin{itemize}
\item[(i)] $\wt=h_--h_+$.
\item[(ii)] $\tilde{e}c=0$ if and only if $h_+(c)=0$. Similarly, $\tilde{f}c=0$ if and only if $h_-(c)=0$.
\item[(iii)] For $c,c'\in C$ the equality $\tilde{e}c=c'$ is equivalent to $\tilde{f}c'=c$.
\item[(iv)] If $c'=\tilde{e}c$, then $h_+(c')=h_+(c)-1, h_-(c')=h_-(c)+1$.
\end{itemize}

Let $C,C'$ be $\sl_2$-crystals. A map $\varphi:C\rightarrow C'$ is said to be a morphism
of crystals if it intertwines $h_-,h_+,\tilde{e},\tilde{f}$, i.e., $h_-(\varphi(c))=h_-(c),
h_+(\varphi(c))=h_+(c), \wt(\varphi(c))=\wt(c), \tilde{e}\varphi(c)=\varphi(\tilde{e}c),
\tilde{f}\varphi(c)=\varphi(\tilde{f}c)$ (of course, we set $\varphi(0)=0$).
In particular, if $C\subset C'$ and the inclusion
is a morphism of crystals one says that $C$ is a subcrystal of $C'$.


\begin{Ex}\label{Ex:main_crystal}
The following example will be of great importance. Consider the set $C=\{+,-\}^n$.
So an element of $C$ is an ordered $n$-tuple of $+$'s and $-$'s. Let us define
the {\it reduced form} of an element $t\in C$ as follows. This will be an
$n$-tuple whose elements are $+,-$ or $0$. We transform $t=(t_1,\ldots,t_n)$ step by step
as follows: for any $a<b$ with $t_a=-, t_b=+, t_i=0$ for $a<i<b$ we replace
$t_a,t_b$ with $0$'s. We continue these transformations while possible and so
we stop when in $t$ no $+$ appears to the right of a $-$. This is the reduced form
of interest to be denoted by $t^{red}$. It is easy to check that $t^{red}$ is well-defined.
Also we remark that $t^{red}_i=t_i$ or $0$.
Define $h_+(t), h_-(t)$ as the number of $+$'s and $-$'s in the reduced
form of $t$. Further, let $\tilde{e}t$ be the sequence obtained from $t$ by
changing $t_i$ from $+$ to $-$, where $i$ is the largest index such that $t^{red}_i=+$.
We set $\tilde{e}t=0$ if no such index $i$ exists.
Similarly, let $\tilde{f}t$ be the sequence obtained from $t$ by changing $t_j$
from $-$ to $+$, where $j$ is the smallest index such that $t^{red}_j=-$.
We set $\tilde{f}t=0$ if no such index $j$ exists.
It is straightforward to check that $C$ together with these structures is a crystal.
%
\end{Ex}

In fact, to any $\sl_2$-categorification $\Cat$ one can assign a crystal in a standard way.
Namely, $C$ is the set of simples, the functions $h_+,h_-$
are as defined in  \ref{SS_SL2_prop} and $\wt(S)=a$ if $S\in \Cat_a$. Further, $\tilde{e}S$
is 0 if $ES=0$, and $\tilde{e}S$ is a unique simple object in the socle (equivalently, in the head) of $ES$ if $ES\neq 0$,
see Proposition \ref{Prop:simples}.
The map $\tilde{f}$ is defined similarly using $F$ instead of $E$. The condition (iii) follows
from  $\Hom_\Cat(ES,T)=\Hom_{\Cat}(S,FT)$.


Similarly, one can introduce a crystal for $\g(Q)$ and produce such a crystal from a $\g(Q)$-categorification.
We will not need this.

\section{Highest weight categories}\label{S_hw}
\subsection{General definition}
As before, $\C$ stands for a field.
Recall that by a highest weight category one means a pair $(\Cat,\Lambda)$ of an artinian $\C$-linear abelian category $\Cat$
and a poset $\Lambda$ equipped with a collection of objects $\Delta(\lambda)\in \Cat$, one for each $\lambda\in \Lambda$.
These data are supposed to satisfy the following conditions.
\begin{itemize}
\item[(HW1)] $\End_{\Cat}(\Delta(\lambda))=\C$.
\item[(HW2)] There is a unique simple quotient $L(\lambda)$ of $\Delta(\lambda)$,
and each simple in $\Cat$ is isomorphic to precisely one $L(\lambda)$.
\item[(HW3)] For each $\lambda\in \Lambda$  there is a projective object $P(\lambda)$
equipped with a filtration $P(\lambda)=F_0\supset F_1\supset F_2\ldots$
such that $F_0/F_1=\Delta(\lambda)$ and $F_i/F_{i+1}=\Delta(\lambda_i)$ with $\lambda_i>\lambda$ for all $i>0$.
\item[(HW4)] The BGG reciprocity holds: for all $\lambda,\mu\in \Lambda$ the multiplicity of $L(\mu)$ inside $\Delta(\lambda)$
equals to the multiplicity of $\Delta(\lambda)$ inside $P(\mu)$.
\end{itemize}

For a highest weight category $\Cat$ let $\Cat^\Delta$ denote the full exact subcategory of $\Delta$-filtered
objects, i.e., those that have a filtration whose successive quotients are of the form $\Delta(\lambda),\lambda\in \Lambda$.

Now let us recall the definition of costandard objects. Following \cite[Proposition 4.19]{rouqqsch}, this is  a unique
set of objects $\nabla(\lambda)$ indexed by $\Lambda$ such that $(\Cat^{opp}, \nabla(\lambda))$ is
a highest weight category, and $\Ext^i(\Delta(\lambda),\nabla(\mu))=\C$ if $i=0, \lambda=\mu$
and $0$ else. Recall, \cite[Lemma 4.21]{rouqqsch} that an object $N\in \Cat$ lies in $\Cat^\Delta$
if and only if $\Ext^1(N,\nabla(\lambda))=0$ for all $\lambda\in \Lambda$.

\subsection{Representations of $\GL$}\label{SS_GL}
Assume from now on that the field $\C$ is algebraically closed.
Consider the category $\Cat:=\operatorname{Rep}(\GL_n(\C))$ of all rational finite dimensional
representations of $\GL_n(\C)$. This is a highest weight category: the standard objects are
the Weyl modules $\Delta(\lambda)$, where $\lambda$ is a strictly dominant weight, i.e., $\lambda=(\lambda_1,\ldots,\lambda_n)$,
where $\lambda_1> \lambda_2>\ldots> \lambda_n$ are  integers,
see \cite{Jantzen} for details (the highest weight of $\Delta(\lambda)$
is $\lambda-\rho$, where $\rho:=(0,-1,-2,\ldots,1-n)$). For an ordering  we can take the usual  ordering on
the dominant weights: $\lambda\leqslant \mu$ if  $\mu-\lambda$ is a linear combination of roots
with non-negative integral coefficients.

The categorification structure is introduced as follows, see \cite[7.5]{CR}.
Consider the tensor Casimir $\Omega=\sum_{i,j}e_{ij}\otimes e_{ji}\in \gl_n\times\gl_n$, where $e_{ij}$ is the unit matrix $(\delta_{ki}\delta_{lj})_{l,k=1}^d$. For $M\in \Cat$ and $i\in \Z$ let $F_iM$ (resp., $E_i M$) be the $i$-th (resp., $-n-i$-th) generalized eigenspace of $\Omega$ on $\C^n\otimes M$ (resp., on $(\C^n)^*\otimes M$). Of course, $E_i=E_j, F_i=F_j$ if $i=j$ in $\C$.
So we get a $\gl_\infty$-categorification if $\operatorname{char}\C=0$ or an $\hat{\sl}_p$-categorification if
$\operatorname{char}\C=p$.

Let us now explain the categorification structure on the polynomial representations of $\GL$, see \cite{HY}
for details. Let $\operatorname{Rep}^d(\GL_n(\C))$ denote the subcategory in $\operatorname{Rep}(\GL_n(\C))$
consisting of all polynomial representations of degree $d$. The simple $L(\lambda)$ (equivalently, the
standard $\Delta(\lambda)$) is polynomial of degree $d$ if and only if $\lambda_n\geqslant 0, \sum_{i=1}^n \lambda_i=d$.
If $\lambda<\mu$ and $L(\mu)$ is a polynomial representation, then so is $L(\lambda)$. It follows that
$\operatorname{Rep}^d(\GL_n(\C))$ is a highest weight category, whose standard objects are still the Weyl
modules. The categories $\operatorname{Rep}^d(\GL_n(\C))$ are mutually equivalent (as highest weight categories)
as long as $n\geqslant d$. Any of these categories is denoted by $\operatorname{Rep}^d(\GL(\C))$.
In \cite{HY} it was shown how to modify the construction above to produce a categorification of $\bigoplus_{d=0}^\infty
\operatorname{Rep}^d(\GL(\C))$ (a modification is necessary because $E_i M$ may not be polynomial if $M\in \operatorname{Rep}^d(\GL_n(\C))$).
We remark that the objects in $\bigoplus_{d=0}^\infty \operatorname{Rep}^d(\GL(\C))$ are naturally parameterized
by Young diagrams.

\subsection{Parabolic categories $\OCat$}
Assume $\C$ has characteristic $0$.
Pick a positive integer $n$ and consider the Lie algebra $\g=\gl_n$.
Next, pick a collection $\underline{n}:=(n_1,\ldots,n_k)$ of positive integers summing to
$n$. Let $e_1,\ldots,e_n$ be the tautological basis in $\C^n$. Let
$\p$  be the parabolic subalgebra that stabilizes the subspaces
$$\Span(e_1,\ldots,e_{n_1+n_2+\ldots+n_i}), \quad i=1,\ldots,k$$
and let $\lf$ be its Levi subalgebra preserving the subspaces
$$\Span(e_{n_1+\ldots+n_{i-1}+1},\ldots, e_{n_1+\ldots +n_i}),\quad i=1,\ldots,k.$$

Consider the category $\OCat^{\underline{n}}$ consisting of all $\g$-modules with integral central
characters, where $\p$ acts locally finitely and $\lf$ acts semisimply. This is a highest
weight category, where standard objects are parabolic Verma modules $\Delta(\lambda)$ with
$\lambda$ being a {\it parabolically strictly dominant} weight in the sense that
$\lambda_1>\ldots>\lambda_{n_1},\lambda_{n_1+1}>\ldots>\lambda_{n_2},\ldots, \lambda_{n_1+\ldots+n_{k-1}+1}>\ldots>\lambda_n$.
An ordering on $\OCat^{\underline{n}}$ is chosen as in Subsection \ref{SS_GL}.

The category $\OCat^{\underline{n}}$ comes equipped with a $\gl_{\infty}$-categorification,
see \cite[7.4]{CR}, analogously to Subsection \ref{SS_GL}.
A similar construction works for parabolic categories $\mathcal{O}_\epsilon^{\underline{n}}$ for the Lusztig form $U_\epsilon(\gl_n)$
of the quantized enveloping algebra at a root of unity $\epsilon$. There we get a categorical $\hat{\sl}_m$-action
on the category $\mathcal{O}^{\underline{n}}_{\epsilon}$, where $m$ is the order of $\epsilon$.

\subsection{Cherednik categories $\OCat$: general case}\label{SS_Ch}
We are going to start by recalling the definition of the Rational Cherednik algebras due to Etingof and Ginzburg, \cite{EG}.
In this and a subsequent subsection we are going to assume that $\C$ is the field of complex numbers.

Let $\h$ be a complex vector space, and $W\subset \GL(\h)$ be a finite subgroup generated
by complex reflections. Recall that $s\in \GL(\h)$ is called a complex reflection if the dimension
of the fixed point subspace $\h^s$ equals $\dim \h-1$. Let $S_0,\ldots,S_r$ be all conjugacy classes
of complex reflections in $W$. For each $S_i$ pick a complex number $c_i$. Also for a complex reflection
$s$ let $\alpha_s\in \h^*, \alpha_s^\vee\in \h$ be elements vanishing on $\h^s, (\h^*)^s$, respectively,
normalized by $\langle\alpha_s,\alpha_s^\vee\rangle=2$.

Set $p:=(c_0,\ldots,c_r)$. The rational Cherednik algebra $H_p(=H_p(\h,W))$ is the quotient of the smash-product
$T(\h\oplus\h^*)\# W$ (=the semidirect tensor product of $\C W$ and $T(\h\oplus\h^*)$) by the relations
\begin{equation}\label{eq:Chered}
[x,x']=[y,y']=0, [y,x]=1- \sum_{i=0}^r c_i \sum_{s\in S_i} \langle x,\alpha_s^\vee\rangle\langle y,\alpha_s\rangle.
\end{equation}

Following \cite[4.2]{GGOR} consider the map $w\mapsto w^{-1}:W\rightarrow W$. It induces an involution on the set of conjugacy classes and hence an involution $p\mapsto p^*$ on the set of parameters. The map $x\mapsto x, y\mapsto -y, w\mapsto w^{-1}, x\in \h^*, y\in \h, w\in W$
gives rise to an isomorphism

\begin{equation}\label{eq:dual_iso} H_p(\h,W)\xrightarrow{\sim} H_{p^*}(\h^*,W)^{opp}.\end{equation}

Let us proceed now to the category $\OCat$. According to \cite{EG},  the natural homomorphism
$S(\h)\otimes \C W\otimes S(\h^*)\rightarrow H_c$ is a bijection. So, following \cite{GGOR}, we can consider the category $\OCat_p(W)(=\OCat_p(\h,W))$
consisting of all finitely generated $H_c$-modules, where the action of $\h$ is locally nilpotent. We remark that
any module in $\OCat_p(W)$ is finitely generated over $\C[\h]=S(\h^*)$. Let us give an important example of a module
in $\OCat_p(W)$. Take an irreducible $W$-module $L$ and set $\Delta(L):=(\Delta_p(L)=)H_p\otimes_{S(\h^*)\# W} L$, where $\h^*$
is supposed to act by $0$ on $L$. This is a so called standard (or Verma) module in $\OCat_p(W)$.

It turns out that $\OCat_p(W)$ together with the collection of standard modules is a highest weight category.
A partial order on $\operatorname{Irr}(W)$ can be defined as follows. Consider the {\it deformed Euler element}
$\mathbf{eu}\in H_c$ given by
\begin{equation}\label{eq:Euler}
\mathbf{eu}:=\sum_{i=1}^{\dim \h} x_iy_i + \frac{\dim \h}{2} - \sum_{i=0}^ r c_i \sum_{s\in S_i} \frac{2}{1-\lambda_s}s.
\end{equation}
Here $x_i,y_i$ are mutually dual bases of $\h^*,\h$ and  $\lambda_s$ denotes the only non-unit eigenvalue of $s$
in its action on $\h^*$. The element $\mathbf{eu}$ commutes with $W$, while $[\mathbf{eu},x]=x, [\mathbf{eu},y]=-y$
for $x\in \h^*,y\in \h$. Given a parameter $p$ define a {\it $c$-function} $c_p:\operatorname{Irr}(W)\rightarrow \C$ as follows: $c_p(E)$
is the difference between the eigenvalues of $\mathbf{eu}$ on $E\subset \Delta(E)$ and on $\operatorname{triv}\subset \Delta(\operatorname{triv})$. We set $E>_pE'$ if $c_p(E)-c_p(E')\in \mathbb{Z}_{<0}$. Clearly, if $L(E')$ appears in $\Delta(E)$, then $E>_p E'$.

An important tool to study the category $\OCat_p(W)$ is the KZ functor from \cite{GGOR}. It is a surjective
exact functor from $\OCat_p(W)$ to the category of modules over the Hecke algebra $\Hecke_p(W)$, whose parameters
are recovered from $p$. An important property of the KZ functor is that it is fully faithful on projectives.


Now let us recall the duality for the categories $\OCat$ from \cite[4.2]{GGOR}. Take a module $M\in \OCat_p(W)$.
The space $\Hom_{\C}(M,\C)$ has a natural structure of a right $H_p(\h,W)$-module and thanks to
(\ref{eq:dual_iso}) of a left $H_{p^*}(\h^*,W)$-module. Let $D(M)$ be the span of all generalized
$\mathbf{eu}$-eigenvectors in $\Hom_{\C}(M,\C)$. This is a module in $\OCat_{p^*}(\h^*,W)$. It is easy
to show that $D^2$ is the identity functor.
The costandard objects in $\OCat_p(W)$ are given by $\nabla_p(E)=D(\Delta_{p^*}(E^*))$.

\subsection{Cherednik categories $\OCat$: cyclotomic case}\label{SS_Ch_cycl}
The categories $\mathcal{O}$ for general Rational Cherednik algebras do  not give rise to categorifications.
The latter appear only in the {\it cyclotomic} case that we are going to describe now.

Suppose $W=G(\ell,1,n)$, where $\ell>1,n\geqslant 1$, is the wreath
product of $S_n$ and the group $\mu_\ell$ of $\ell$-th roots of $1$. That is,  $W:=S_n\ltimes \mu_\ell^n$ acts
on $\h:=\C^n$ in a natural way. For $n>1$ there are following $\ell$ classes of complex reflections in $W$:
\begin{itemize}
\item $S_0$ consisting of elements of the form $(ij)\gamma_i\gamma_j^{-1}$, where $(ij)$ is the transposition
in $S_n$ swapping $i$ and $j$, and $\gamma_i,\gamma_j$ are elements in the $i$-th and $j$-th copies of $\mu_\ell$
inside of $\mu_\ell^n$,
\item $S_i, i=1,\ldots,\ell-1,$ consisting of the elements $\gamma_i$ with $\gamma=\exp(2\pi\sqrt{-1}j/\ell), j=0,1,\ldots,\ell-1$.
\end{itemize}

We remark that for $n=1$ there are $\ell-1$ conjugacy classes:  $S_0$ is absent.

In fact, it is convenient to use another set of parameters. Pick a complex number $\kappa$ and set $c_0:=-\kappa$. The case $\kappa=0$ is non-interesting because the algebra $H_p$ in this case decomposes as $ (H^1_p)^{\otimes n}\# S_n$, where $H^1_p$ is the similar
algebra for $n=1$. Below we always assume that $\kappa\neq 0$. Also let $s_0,\ldots,s_{\ell-1}$ be complex numbers. Then we set
\begin{equation}\label{eq:ss}
c_i:=-\frac{1}{2}(1+\kappa\sum_{j=1}^{\ell-1}\left(\exp(-ij\cdot 2\pi\sqrt{-1}/\ell)-1\right)(s_j-s_{j-1})), i=1,\ldots,\ell-1.
\end{equation}
We remark that two collections ${\bf s}=(s_0,\ldots,s_{\ell-1}), {\bf s}':=(s_0,\ldots,s_{\ell-1})$ give rise to the
same parameters $c_1,\ldots,c_{\ell-1}$ if and only if $s'_i-s_i$ is independent of $i$.

Let us proceed to the category $\mathcal{O}$. In this case we will write $\OCat_p(n)$ instead of $\OCat_p(W)$.

First, let us recall the classical combinatorial description
of $\operatorname{Irr}(W)$: the $W$-irreducibles are parameterized by $\ell$-multipartitions $\lambda:=(\lambda^{(0)},\ldots,\lambda^{(\ell-1)})$
of $n$. Namely, consider the subgroup $G(\lambda,1,\ell)=\prod_{i=0}^{\ell-1} G(|\lambda^{(i)}|,1,\ell)\subset G(n,1,\ell)$.
Here and below for a partition $\mu=(\mu_1,\mu_2,\ldots)$ we set $|\mu|:=\sum_i \mu_i$.
View $\lambda^{(i)}$ as a representation of $S_{|\lambda^{(i)}|}$. Further, let $\lambda^{(i)}(r)$ denote the representation
of $G(|\lambda^{(i)}|,1,\ell)$, that coincides with $\lambda^{(i)}$ as an $S_{|\lambda^{(i)}|}$-module, while for $\gamma\in\mu_\ell$
the element $\gamma_i$ acts by the scalar $\gamma^r$. The irreducible $W$-module corresponding to $\lambda$
is induced from the $G(\lambda,1,\ell)$-module $\lambda^{(0)}(0)\boxtimes\lambda^{(1)}(1)\boxtimes\ldots \boxtimes \lambda^{(\ell-1)}(\ell-1)$.

Let us now provide a formula for the $c$-function, obtained in \cite{rouqqsch}. We will express the $c$-function $c_p(\lambda)$
in terms of the presentation $p=(\kappa, s_0,\ldots,s_{\ell-1})$. We represent a partition $\mu$ as a Young diagram
with $\mu_1$ boxes in the first row, $\mu_2$ in the second row and so on. For a box $x$ lying in the $a$th row and
$b$th column of $\lambda^{(i)}$ we define its ${\bf s}$-{\it shifted content} by $\cont^{\bf s}(x):=s_i+b-a$.
Then set
\begin{align}\label{eq:dfn_d_p}
d^p(x)&:= \kappa \ell \cont^{{\bf s}}(x)-i-\kappa\sum_{i=0}^{\ell} s_i,\\\label{eq:def_c_fn}
c_p(\lambda)&:=\sum_{x\in \lambda} d^p(x).
\end{align}
Up to a scalar independent of $p$ and $n$ the function $c_p(\lambda)$ coincides with the $c$-function introduced
above, see \cite[(2.3.8)]{GL}.

The Hecke algebra $\Hecke_p(W)$ is just the cyclotomic Hecke algebra $\Hecke_q^{\underline{Q}}(n)$ mentioned in
Subsection \ref{SS_sl2_def}. The parameters $q,Q_0,\ldots,Q_{\ell-1}$ are determined by $p=(\kappa, s_0,\ldots,s_{\ell-1})$
in the following way: $q:=\exp(2\pi\sqrt{-1}\kappa), Q_i:=\exp(2\pi\sqrt{-1}\kappa s_i)$.

Now we are in position to recall Shan's categorification, \cite{Shan}.
Consider the category $\OCat_p:=\bigoplus_{n\geqslant 0} \OCat_p(n)$, where $\OCat_p(0)$ is just the category
of finite dimensional vector spaces.

Etingof and Bezrukavnikov defined induction and restriction functors for rational Cherednik algebras in
\cite{BE}. Namely, let $W$ be an arbitrary  complex reflection group acting on $\h$, and let $\underline{W}$
be its parabolic subgroup. Let $\underline{\h}$ be a unique $\underline{W}$-stable complement to
the fixed point subspace $\h^{\underline{W}}$ in $\h$. Abusing the notation we denote the restriction
of $p$ to $S\cap \underline{W}$ again by $p$. According to \cite{BE} there are exact functors
$\Res_W^{\underline{W}}: \OCat_p(\h,W)\rightarrow \OCat_p(\underline{\h},\underline{W})$ and $\Ind_W^{\underline{W}}:\OCat_p(\underline{\h},\underline{W})
\rightarrow \OCat_p(\h,W)$. Shan in \cite{Shan} (see also \cite{fun_iso}) checked that these functors are biadjoint.

In particular, we can consider the case $\underline{W}=G(n-1,1,\ell)\subset G(n,1,\ell)$.
Their construction yields exact endofunctors $\Res, \Ind$ of $\OCat_p$ with $\Res:\OCat_p(n)\rightarrow
\OCat_p(n-1), \Ind:\OCat_p(n-1)\rightarrow \OCat_p(n)$. As Shan checked in \cite{Shan}, the KZ functors
intertwine the inductions/restrictions in $\OCat_p$ with the cyclotomic Hecke algebra inductions/restrictions.
Since the KZ functors are fully faithful on projectives, this yields the generalized eigen-functor decompositions
$\Res=\bigoplus_{z\in \C}E_z, \Ind=\bigoplus_{z\in \C}F_z$. The functors $E_z,F_z$ constitute an $\sl_2$-categorification
on $\OCat_p$.

We are interested in the behavior of the functors $E_z,F_z$ on the K-group $[\OCat_p]$. We have the basis $[\Delta(\lambda)]$ of $[\OCat_p]$
indexed by all multi-partitions $\lambda$. We call a box $x$ of a multipartition $\lambda$ a $z$-box if $\exp(2\pi\sqrt{-1}\kappa
\cont^{\bf s}(x))=z$. Recall that we represent $\lambda$ as a collection of Young diagrams. A box lying in $\lambda$
(i.e., in one of the diagrams $\lambda^{(i)}$) is called {\it removable} if $\lambda^{(i)}\setminus \{x\}$ is still a
Young diagram (in other words, a box in $j$-th row and $k$th column is removable if $k=\lambda^{(i)}_j$
and $\lambda^{(i)}_{j+1}<\lambda^{(i)}_j$). Similarly, a box lying outside of $\lambda^{(i)}$ is {\it addable} if
$\lambda^{(i)}\sqcup \{x\}$ is again a Young diagram (equivalently, $x$ lies in $j$th row and $\lambda^{(i)}_j+1$th column
with either $j=1$ or $\lambda^{(i)}_j<\lambda^{(i)}_{j-1}$).

\cite[Proposition 4.4]{Shan} has the following straightforward generalization.
\begin{Prop}\label{Prop:Kgroup_action}
We have $[E_z\Delta(\lambda)]=\bigoplus_x [\Delta(\lambda\setminus \{x\})]$, where the sum is taken over all
removable $z$-boxes $x$. Further, $[F_z\Delta(\lambda)]=\bigoplus_x [\Delta(\lambda\sqcup\{x\})]$, where the sum
is taken over all addable $z$-boxes $x$.
\end{Prop}

\section{Highest weight $\sl_2$-categorifications}\label{S_hw_sl2_cat}
\subsection{Definition}\label{SS_hwc_def}
Assume now that $(\Cat,\Lambda)$ is  a highest weight category and that $\Cat$ is equipped
with an $\sl_2$-categorification, let $E,F$ be the categorification
functors. We say that $\Cat$ is a highest weight categorification if there are
\begin{itemize}
\item a function $c:\Lambda\rightarrow \mathbb{C}$,
\item an index set $\mathfrak{A}$, a collection of non-negative integers $n_a, a\in \mathfrak{A},$
a partition $\Lambda=\bigsqcup_{a\in \mathfrak{A}} \Lambda_a$,
\item identifications $\sigma_a: \{+,-\}^{n_a}\xrightarrow{\sim} \Lambda_a$, and functions $d_a:\{1,2,\ldots,n_a\}\rightarrow \mathbb{C}$
\end{itemize}
such that the following conditions are satisfied.

\begin{itemize}
\item[(HWC0)] The functors $E,F$ preserve the subcategory $\Cat^{\Delta}$ of $\Delta$-filtered
objects.
\item[(HWC1)] The inequality $\lambda<\mu$ implies $c(\lambda)>c(\mu)$.
\item[(HWC2)] For $a\in \mathfrak{A},t\in \{+,-\}^{n_a}$, in the $K$-group of $\Cat$ we have
$e[\Delta(\sigma_a(t))]=\sum_j [\Delta(\sigma_a(t^j))]$, where the sum
is taken over all $j\in \{1,2,\ldots,n_a\}$ such that $t_j=+$, where $t^j\in \{+,-\}^{n_a}$ is given by $t^j_k=t_k$
for $k\neq j$ and $t^j_j=-$. Similarly, $f[\Delta(\sigma_a(t))]=\sum_l [\Delta(\sigma_a(\bar{t}^l))]$, where the sum
is taken over all $l\in \{1,2,\ldots,n_a\}$ such that $t_l=-$, and $\bar{t}^j_k=t_k$
for $k\neq j$ and $\bar{t}^j_j=+$.
\item[(HWC3)] For $t^j,\bar{t}^l$ as above we have $c(t^j)=c(t)+d_a(j), c(\bar{t}^l)=c(t)-d_a(l)$.
\item[(HWC4)] Finally, for any $a$ we have $d_a(1)<d_a(2)<\ldots<d_a(n_a)$.
\end{itemize}

Recall that for complex numbers $\alpha,\beta$ we write $\alpha<\beta$ if $\beta-\alpha$ is a positive integer.
The definition is, of course, obtained by generalizing  examples.

We remark that for further results to be obtained in \cite{hw_cat_str} we will use a finer ordering
on $\Cat$ and so will need to modify the definition of a highest weight categorification making it
much more technical.


\subsection{Examples}
In this section we will show that all examples of $\sl_2$-categorifications we considered
before are actually highest weight categorifications in the sense of Subsection \ref{SS_hwc_def}.
We are going to consider the Cherednik case  in detail and only
sketch  the other (that are more standard).

Let $\OCat_p$ be the sum $\bigoplus_{n\geqslant 0}\OCat_p(n)$ of the categories
$\mathcal{O}$ for the cyclotomic Cherednik algebra $H_p(n)$
as in Subsection \ref{SS_Ch_cycl}. Assume that $\kappa$ is not integral.  
We set $\Cat:=\bigoplus_{n\geqslant 0}\OCat_p(n), E:=F_z, F:=E_z$, where $z$ is some
complex number. Then $E,F$ define an $\sl_2$-categorification on $\OCat_p$.

\begin{Lem}
The $\sl_2$-categorification $\Cat$ satisfies (HWC0).
\end{Lem}
\begin{proof}
From the definition of the functors $E_z,F_z$ it follows that one only needs to prove that
all Bezrukavnikov-Etingof functors $\Res_W^{\underline{W}}, \Ind_W^{\underline{W}}$ preserve $\Cat^{\Delta}$.  By \cite[Proposition 1.9]{Shan},
the restriction functors $\Res_W^{\underline{W}}$ preserves $\Cat^{\Delta}$. Recall that $\Ind_W^{\underline{W}}$ is right
adjoint to $\Res_W^{\underline{W}}$ and both $\Ind_W^{\underline{W}},\Res_W^{\underline{W}}$ are exact. So for any $M,N\in \Cat$ we $\Ext^i(\Res_W^{\underline{W}}(M),N)=\Ext^i(M,\Ind_W^{\underline{W}}(N))$.
Now \cite[Lemma 4.21]{rouqqsch} implies that $\Ind_W^{\underline{W}}$ preserves the subcategory $\Cat^{\nabla}\subset \Cat$
of all costandardly filtered objects. Recall  that Bezrukavnikov and Etingof, \cite{BE}, introduced functors $\res_W^{\underline{W}}:=D\circ \Res_W^{\underline{W}}\circ D,\ind_W^{\underline{W}}:=
D\circ \Ind_W^{\underline{W}}\circ D$, where $D$ is the duality functor recalled in Subsection \ref{SS_Ch}.
Since the standard and costandard objects are related via $D$, we see that $\res_W^{\underline{W}}$ preserves $\Cat^{\nabla}$,
while $\ind_W^{\underline{W}}$ preserves $\Cat^{\Delta}$. But, according to \cite{fun_iso}, $\Res_W^{\underline{W}}\cong\res_W^{\underline{W}}$
and $\Ind_W^{\underline{W}}\cong \ind_W^{\underline{W}}$. This completes the proof.
\end{proof}

Let us explain the choice of $c,\mathfrak{A},n_a, \Lambda_a, \sigma_a, d_a$ making $\Cat$ into a highest weight categorification.
For $c$ we just take the $c$-function recalled in Subsection \ref{SS_Ch}. The condition (HWC1) follows. Two multipartitions $\lambda,\mu$
belong to the same set $\Lambda_a$ if the multipartitions obtained from $\lambda$ and $\mu$ by removing all removable
$z$-boxes coincide. The following easy combinatorial lemma shows that the sets of addable and removable
$z$-boxes in $\lambda$ and $\mu$ coincide.

\begin{Lem}\label{Lem:partition}
For a (multi)partition $\mu$ let $B_z(\mu)$ denote the set of all addable and removable $z$-boxes in
$\mu$. Then for any addable $z$-box $x$ we have $B_z(\mu\sqcup x)=B_z(\mu)$.
\end{Lem}
\begin{proof}
It is easy to see that adding a $z$-box affects only the sets $B_{q^{\pm 1}z}(\mu)$, where $q=\exp(2\pi\sqrt{-1}\kappa)$.
\end{proof}

For $n_a$ we take the cardinality of $B_z(\lambda),\lambda\in \Lambda_a,$ from the previous
lemma. For a $z$-box $x$ of $\lambda$  we set $d_a(x):=d^p(x)$ (where the last number is equal to $\kappa \ell \cont^{s_r}(x)-r-\kappa\sum_{i=0}^{\ell-1}s_i$
if $x$ is in $\lambda^{(r)}$). If $x\in \lambda^{(r)},y\in \lambda^{(r')}$ are $z$-boxes, then $\kappa\cont^{s_r}(x)-\kappa \cont^{s_{r'}}(y)\in \mathbb{Z}$. Also there is at most one addable/removable box with a given content in each diagram. So for different
$x,y\in B_z(\lambda)$ the numbers $d_a(x),d_a(y)$ differ by a nonzero integer. Let us number boxes $x_1,\ldots,x_{n_a}$ so that
the sequence $d_a(x_j)$ increases. Now we can define the map $\sigma_a:\{+,-\}^{n_a}\xrightarrow{\sim} \Lambda_a$.
By definition, it sends an $n_a$-tuple $t$ to the only multipartition  $\lambda(t)\in \Lambda_a$, where the box $x_i$ is in $\lambda(t)$
if and only if $t_i=-$ (and so it is removable, the boxes $x_i$ with $t_i=+$ are addable). (HWC2) follows now from Proposition \ref{Prop:Kgroup_action}. Finally, set $d_a(j):=d_a(x_j)$. (HWC4) is tautological, and (HWC3) is a consequence of (\ref{eq:dfn_d_p},\ref{eq:def_c_fn}). So we have checked that $\Cat$ is a highest weight categorification.

Let us briefly outline the other examples. Let $\C$ be an algebraically closed field.
Consider the category $\operatorname{Rep}(\GL_n)$
and fix an integer $i$. Set $E:=F_i, F:=E_i$.
Let $\lambda=(\lambda_1,\ldots,\lambda_n)$ be a highest weight.
Let $I_i(\lambda)$ be the subset of $\{1,\ldots,n\}$ consisting of all
indexes $j$ with $\lambda_j-i=0$ in $\C$.  Then $F_i\Delta(\lambda)$
has a filtration whose successive quotients  are the Weyl modules
$\Delta(\lambda+\epsilon_j), j\in I_i(\lambda),$  appearing with multiplicity $1$ if
$\lambda+\epsilon_j$ is dominant (and with multiplicity 0 else). Similarly,
$E_i\Delta(\lambda)$ has a filtration whose successive quotients  are the Weyl modules
$\Delta(\lambda-\epsilon_j), j\in I_{i+1}(\lambda),$  appearing with multiplicity $1$ if
$\lambda-\epsilon_j$ is still in $\Lambda$. See \cite[Theorems A,A']{BK_functors} for details. This shows (HWC0).
For the $c$-function we take $c(\lambda)=\sum_{i=1}^n i\lambda_i$ which implies (HWC1).
Two dominant weights $\lambda$ and $\mu$ lie in the same $\Lambda_a$ if
for each $j=1,\ldots,n$ exactly one of the following possibilities holds:
\begin{itemize}
\item $\lambda_j=\mu_j$ in $\Z$.
\item $\lambda_j=\mu_j+1$ in $\Z$ and $\mu_j=i$ in $\C$.
\item $\mu_j=\lambda_j+1$ in $\Z$ and $\lambda_j=i$ in $\C$.
\end{itemize}
Let $j_1<j_2<\ldots<j_{n_a}$ be all indexes $j$ such that $\lambda_j\neq \mu_j$ for $\lambda,\mu\in \Lambda_a$.
The $n_a$-tuple $t:=\sigma_a^{-1}(\lambda)$ for $\lambda\in \Lambda_a$ is given by $t_l:=+$ if $\lambda_{j_l}=i$
in $\C$ and $t_l=-$ if $\lambda_{j_l}=i+1$ in $\C$. We set $d_a(l):=j_l$. Now it is easy to verify the remaining axioms
(HWC2)-(HWC4).

For the remaining two categories, $\bigoplus_{d=0}^\infty \operatorname{Rep}^d(\GL)$ and $\mathcal{O}^{\underline{n}}$,
one introduces the additional structures and checks  (HWC0)-(HWC4) hold in a similar way (using
\cite{HY} and \cite{BKshift} instead of \cite{BK_functors}).

\section{Structure of the crystal}\label{S_crystal}
\subsection{Main result}
Suppose that $(\Cat,\Lambda)$ is a highest weight $\sl_2$-categorification in
the sense of Subsection \ref{SS_hwc_def}.  Recall
the index set $\mathfrak{A}$, the subsets $\Lambda_a\subset \mathfrak{A}$, the integers $n_a$,
and the bijections $\sigma_a:\{+,-\}^{n_a}\xrightarrow{\sim} \Lambda_a$ introduced in
Subsection \ref{SS_hwc_def}. Recall the crystal structure on $\{+,-\}^{n_a}$ introduced in
Example \ref{Ex:main_crystal}.

\begin{Thm}\label{Thm:crystal_main}
For each $a\in \mathfrak{A}$ the set  $\{L(\lambda), \lambda\in \Lambda_a\}$ is a subcrystal
in $\{L(\lambda), \lambda\in \Lambda\}$ and the
map $t\mapsto L(\sigma_a(t)):\{+,-\}^{n_a}\xrightarrow{\sim} \{L(\lambda), \lambda\in \Lambda_a\}$
is an isomorphism of crystals.
\end{Thm}

We remark that for $\Cat=\operatorname{Rep}(\GL_n)$ Theorem \ref{Thm:crystal_main} gives the same
description of the crystal as \cite{Kleshchev}, see also \cite[Theorems B,B']{BK_functors}, while for $\Cat=\OCat^{\underline{n}}$
we recover \cite[4.3]{BKshift}. We also expect that it is possible to extend the techniques used in the proof  to
the setting of standardly filtered categories and functors more general than  $\sl_2$-categorification
functors.  This should allow to recover the results from \cite{KS}
that give a combinatorial description of the crystal for representations of the supergroup $Q(n)$.

\subsection{Preliminary considerations}
The claim that $\Lambda_a$ is a subcrystal follows from the following proposition.
\begin{Prop}\label{Prop:subcrystal}
Fix $t=(t_1,\ldots,t_{n_a})\in \{+,-\}^{n_a}$. Then either $\tilde{e}L(\sigma_a(t))=0$ or there is $j$ with $t_j=+$ such that $\tilde{e}L(\sigma_a(t))=L(\sigma_a(t^j))$, where
$t^j\in \{+,-\}^{n_a}$  is given by $t^j_l=t_l$ for $l\neq j$ and $t^j_j=-$. Similarly, either $\tilde{f}L(\sigma_a(t))=0$
or there is  $k$ with $t_k=-$ such that $\tilde{f}L(\sigma_a(t))=L(\sigma_a(\bar{t}^k))$, where
$\bar{t}^k\in \{+,-\}^{n_a}$  is given by $\bar{t}^k_l=t_l$ for $l\neq k$ and $\bar{t}^k_k=+$.
\end{Prop}

In the proof we will need the following lemma.

\begin{Lem}\label{Lem:stand_filtr}
We keep the notation of Proposition \ref{Prop:subcrystal}. Let $j_1>j_2>\ldots>j_l$ be all indexes such that $t_{j_i}=+$.
Then there is a filtration $E\Delta(\lambda)=\Fi_0E\Delta(\lambda)
\supset \Fi_1 E\Delta(\lambda)\supset \Fi_2 E\Delta(\lambda)\supset\ldots \supset \{0\}$ such that
$\Fi_{i-1} E\Delta(\lambda)/\Fi_i E\Delta(\lambda)=\Delta(\sigma_a(t^{j_i}))$.
\end{Lem}
\begin{proof}
By (HWC0), $E\Delta(\lambda)\in \Cat^{\Delta}$.
Since the classes of the standard objects form a basis in the K-group, (HWC2) implies that the successive
quotients of a filtration by standards on $E\Delta(\lambda)$ are exactly $\Delta(\lambda^j)$ with $\lambda^j:=\sigma_a(t^j)$
each occurring with multiplicity $1$. By (HWC3),(HWC4), we have $c(\lambda^{j_1})>c(\lambda^{j_2})>\ldots>c(\lambda^{j_l})$
in the ordering of the highest weight category. Now the claim of the lemma follows from (HWC1).
\end{proof}

A filtration from Lemma \ref{Lem:stand_filtr} will be referred to as a {\it standard filtration}.

\begin{proof}[Proof of Proposition \ref{Prop:subcrystal}]
We will prove the first statement, the second one is completely analogous.
Set $\lambda:=\sigma_a(t), \lambda^j:=\sigma_a(t^j)$.

For an object $M\in \Cat$ let $\head(M)$ denote its head, that is, the maximal semisimple quotient. We remark
that $\head$ can be viewed as an  endofunctor of $\Cat$. This functor can be easily seen to be right exact.
Indeed, we can write $\head(M)=\bigoplus_{\lambda\in \Lambda}\Hom_{\Cat}(M,L(\lambda))^*\otimes_{\C}L(\lambda)$.

Of course, $\head(\Delta(\mu))=L(\mu)$ for all $\mu\in \Lambda$. From here, the right exactness of $\head$
and Lemma \ref{Lem:stand_filtr}  we deduce that $\head(E\Delta(\lambda))\subset \bigoplus_j L(\lambda^j)$.

Now we recall that $E$ is an exact functor so $E\Delta(\lambda)\twoheadrightarrow E L(\lambda)$.
The right exactness of $\head$ implies that $\head(E\Delta(\lambda))\twoheadrightarrow \head(EL(\lambda))$.
But the  $\head(EL(\lambda))=\tilde{e} L(\lambda)$, by definition. Thanks to the previous paragraph, we are done.
\end{proof}

Till the end of the section we write $n$ for $n_a$, $\sigma$ for $\sigma_a$. Further,
for $t\in \{+,-\}^n$ we write $\Delta(t):=\Delta(\sigma(t)), L(t):=L(\sigma(t))$.


\begin{Lem}\label{Lem:weights}
We have $\wt(t)=\wt(L(t))$.
\end{Lem}
\begin{proof}
We remark that $L(t)$ lies in $\Cat_i$ if and only if $\Delta(t)\in \Cat_i$.
The inclusion $\Delta(t)\in \Cat_{\wt(t)}$ can be easily deduced from (HWC2).
\end{proof}

\subsection{Ext vanishing}

In this subsection we are going to prove an important technical result. For $t\in \{+,-\}^n$ and $k\in \{1,2,\ldots,n\}$ set
$h^k_-(t):=h_-(t_{k},\ldots,t_n)$. Clearly, $h_-(t)=h^1_-(t)\geqslant h^2_-(t)\geqslant\ldots\geqslant h^n_-(t)$.

We introduce a linear order on $\{+,-\}^n$: we write $t'\succ t$ if there is an index $i$
such that $t_j=t'_j$ for all $j>i$ but $t'_i=-$ and $t_i=+$.

Pick $\mu\in\Lambda$ with $\wt(\mu)=w$ and set $L:=L(\mu)$. Let $m$ be a positive integer such that $F^{m}L=0$.
If $\mu\not\in\Lambda_a$, set $k:=1$. Otherwise, let $k$ be any integer such
that $h^k_-(s)\leqslant m-1$, where $\sigma(s)=\mu$.

\begin{Prop}\label{Prop:vanishing}
Let $L,k,m$ be as in the previous paragraph. If $L=L(s)$,  assume, in addition, that $h_-(L(s'))=h_-(s')$ provided
$\wt(s')<\wt(s)$. Then $\Ext^i(\Delta(t),L)=0$
for $i\leqslant h^k_-(t)-m$.
\end{Prop}

In the proof we will need the following combinatorial lemma.

\begin{Lem}\label{Lem:comb11}
Let $t\in \{+,-\}^n$ and let $l$ be an index with $h^l_-(t)>h^{l+1}_-(t)$.
Let $\bar{t}\in \{+,-\}^n$ be given by $\bar{t}_i=t_i$ for $i\neq l$ and $\bar{t}_l=+$.
Further, let $\Delta(t^1),\ldots, \Delta(t^N)$ be the successive subquotients of the standard filtration on $E\Delta(\bar{t})$ with $t^N\succ t^{N-1}\succ\ldots\succ t^1$. Finally, let $j$ be such that $t^j=t$. Then the following holds
\begin{enumerate}
\item $h_-^{l+1}(t^j)+1=h_-^l(t^j)$.
\item $h_-^{l+1}(t^i)=h_-^{l+1}(t^j)$ for $i<j$.
\item $h_-^{l+1}(t^i)\geqslant h^l_-(t^j)+1$ for $i>j$.
\end{enumerate}
\end{Lem}
\begin{proof}
(1) follows from $h^l_-(t)>h^{l+1}_-(t)$. To prove (2) we just notice that $t^i_k=t^j_k$ for all $k>l$ provided $i<j$.

Let us prove (3). First of all, let us remark that $h_-^{l+1}(t^i)\geqslant h_-^{l+1}(t^{j+1})$ for $i>j$.
So it is enough to consider the case $i=j+1$. Let us note that $t^j_l=-, t^{j+1}_l=+$. Let $p$ be the index with $t^j_{p}=+, t^{j+1}_p=-$.
Then $t^j_k=t^{j+1}_k$ for all $k$ different from $l,p$. Since $t^j_l$ survives in the reduced form, $p>l+1$. Also $p$
is the minimal index bigger than $l$ with $t^j_p=+$ and so $t^j_{l+1}=\ldots=t^j_{p-1}=-$.
Since $t^j_l$ survives in the reduced form, we see that $p-l> h_+(t^j_p,\ldots, t^j_n)=1+h_+(t^j_{p+1},\ldots,t^j_n)$. Moreover, $h_-^{l}(t^j)=p-l-1-h_+(t^j_{p+1},\ldots,t^j_n)+h_-(t^j_{p+1},\ldots,t^j_n)$. Similarly, $h^{l+1}_-(t^{j+1})=p-l-
h_+(t^j_{p+1},\ldots,t^j_n)+h_-(t^j_{p+1},\ldots,t^j_n)$. So $h^{l+1}_-(t^{j+1})=h^l_-(t^j)+1$ and we are done.
%
\end{proof}

\begin{proof}[Proof of Proposition \ref{Prop:vanishing}]
We remark that the claim for $j=0$ just follows from $\sigma(t)\neq \mu$. Indeed, even if $\mu=\sigma(s)$ for some $s$, we have $h^k_-(t)\geqslant m> h^k_-(s)$.

We prove the statement by using the decreasing induction on $l=n,n-1,\ldots,k$ to show that the following holds:
\begin{itemize}
\item[(*)] $\Ext^i(\Delta(t),L)=0$ for $i\leqslant h^{l}_-(t)-m$.
\end{itemize}
The base $l=n$ follows from the previous paragraph.

In the proof we may assume that $h^l_-(t)>h^{l+1}_-(t)$, otherwise we are done by induction. Also we only need
to prove that $\Ext^q(\Delta(t),L)=0$ for $q=h_-^l(t)-m$, the vanishing of the remaining Ext's follows from the
inductive assumptions. We prove the claim in several steps.

{\it Step 1.}
Let $\bar{t},j$ be as in Lemma \ref{Lem:comb11}. Let $\mathcal{F}\supset\mathcal{F}_0$ be the consecutive
filtration subobjects of the standard filtration of $E\Delta(\bar{t})$ such that $\mathcal{F}/\mathcal{F}_0=\Delta(t^j)$.
Let us prove that $\Ext^q(\mathcal{F}, L)=0$. For this consider the exact sequence
$$\Ext^q(E\Delta(\bar{t}),L)\rightarrow \Ext^q(\mathcal{F}, L)\rightarrow \Ext^{q+1}(E\Delta(\bar{t})/\mathcal{F},L).$$
In the next two steps we will prove that the left and the right terms in this sequence vanish that will imply
$\Ext^q(\mathcal{F}, L)=0$.

{\it Step 2.}
Let us prove that $\Ext^q(E\Delta(\bar{t}),L)=0$. By the biadjointness of $E,F$, it is enough to show
that  $\Ext^q(\Delta(\bar{t}),FL)=0$. This will follow if we show that $\Ext^q(\Delta(\bar{t}),L')=0$
for any simple subquotient $L'$ of $FL$. But $F^{m-1}(FL)=0$ and hence $F^{m-1}L'=0$. Also
we have $h^{l+1}_-(\bar{t})=h^l_-(t)-1$ and $h_-^{l+1}(\bar{t})-(m-1)=q$. To complete the proof it
remains to show that, in the case when $L'=L(s')$ for some $s'\in \{+,-\}^n$, we have $h^{l+1}_-(s')\leqslant m-2$. But our assumption
 in the statement of the proposition says $h_-(L')=h_-(s')$. Since $F^{m-1}L'=0$, we get $m-2\geqslant h_-(L')=h_-(s')\geqslant h^{l+1}_-(s')$.   We are done by induction.

{\it Step 3.} Let us prove that $\Ext^{q+1}(E\Delta(\bar{t})/\mathcal{F},L)=0$. The object
$E\Delta(\bar{t})/\mathcal{F}$ inherits the standard filtration from $E\Delta(\bar{t})$. The successive quotients are (in the notation of Lemma \ref{Lem:comb11}) $\Delta(t^i), i>j$.
Now, thanks to assertion (3) of that lemma, for each $i>j$ we have $h^{l+1}_-(t^i)\geqslant h^{l}_-(t)+1$. Therefore, by induction, $\Ext^{q+1}(\Delta(t^i),L)=0$. It follows that $\Ext^{q+1}(E\Delta(\bar{t})/\mathcal{F},L)=0$.

{\it Step 4.} So now we know that $\Ext^q(\mathcal{F},L)=0$. To show that $\Ext^q(\Delta(t),L)=0$ consider the short
exact sequence
$$\Ext^{q-1}(\mathcal{F}_0, L)\rightarrow \Ext^q(\Delta(t^j),L)\rightarrow \Ext^{q}(\mathcal{F},L).$$
It remains to prove that $\Ext^{q-1}(\mathcal{F}_0, L)=0$ and we will do this in the next step.

{\it Step 5.} The object $\mathcal{F}_0$ again  inherits a filtration from $E\Delta(\bar{t})$.
The successive quotients are $\Delta(t^i)$ with $i<j$. According to (1) and (2) of Lemma \ref{Lem:comb11},
$h^{l+1}_-(t^i)=h^{l+1}_-(t^j)=h^l_-(t^j)-1$. So $h^{l+1}_-(t^i)-m=q-1$ and, by induction, we have
$\Ext^{q-1}(\Delta(t^i), L)=0$. Hence $\Ext^{q-1}(\mathcal{F}_0,L)=0$.
\end{proof}

\begin{Rem}
In fact, when $\mu\not\in \Lambda_a$ one can prove that $\Ext^i(\Delta(t),L)=0$ for $i\leqslant h_-(t)-h_-(L)$
(while the proposition above only guarantees $i\leqslant h_-(t)-h_-(L)-1$). We will assume that Theorem \ref{Thm:crystal_main}
holds, this remark is not used to prove it. The proof closely follows that of
the proposition, the only difference is in the proof of Step 2. Namely, let $L'$ be as in that step.
Then, according to (3) of Proposition \ref{Prop:simples}, either $L'=\tilde{f}L$ or $F^{h_-(L)-1}L'=0$.
To prove that $\Ext^q(\Delta(\bar{t}),L')=0$ in the first case we can use the inductive assumption
since still $L'=L(\mu')$ for $\mu'\in \Lambda_a$. In the second case we can apply Proposition
\ref{Prop:vanishing}. 
\end{Rem}

\subsection{Proof of the main theorem}
Let us prove that  $\tilde{f}L(t)=L(\tilde{f}t)$ by using the induction on $w=\wt(t)$.
The case $\wt(t)=-n$ is obvious -- both sides of the equality are zero. Now suppose that the claim is proved for all $s\in \{+,-\}^n$ with
$\wt(s)<w$. This implies $h_-(s)=h_-(L(s))$ provided $\wt(s)<w$. 

We are going to prove, first, that $h_-(s)=h_-(L(s))$ for all $s$ with $\wt(s)=w$.
%
Suppose that $h_-(L(s))<h_-(s)$. Let $\Delta(t^1),\ldots, \Delta(t^N)$ be the successive subquotients of $E\Delta(\tilde{f}s)$ with $t^k\succ t^{k-1}\succ\ldots\succ t^1$, where $\succ$ is the ordering introduced in the beginning of the previous subsection. Let $j$ be such that $t^j=s$. Let $l$ be the index with $(\tilde{f}s)_l=+$ and $s_l=-$. By Lemma \ref{Lem:comb11},
$h_-^{l+1}(t^i)=h_-(t^j)+1$ for $i>j$. So we can apply Proposition \ref{Prop:vanishing} to $L=L(s),t=t^i, m=h_-(s), k=l+1$.
We get $\Ext^1(\Delta(t^i), L(s))=0$. So $L(s)$ is in the head of $E\Delta(\tilde{f}s)$.
Applying Lemma \ref{Lem:obj_subs}, we see that $d(L(s))\geqslant  d(L(\tilde{f}s))$.
But, thanks to the inductive assumption, the right hand side is just $d(\tilde{f}s)=d(s)$. Now recall that $d(?)=\frac{1}{2}(h_-(?)-\wt(?))$.
It follows that $h_-(L(s))=h_-(s)$. This contradicts $h_-(L(s))<h_-(s)$. So we have $h_-(L(s))\geqslant h_-(s)$. But assertion (3) of Proposition \ref{Prop:simples} implies that the number of $s$ with $h_-(L(s))=h$ coincides with the number of $s$ with $h_-(s)=h$
for any $h$. So we see that $h_-(L(s))=h_-(s)$ for any $s$ with $\wt(s)=w$.

Now we are going to prove that $\tilde{e}L(s)=L(\tilde{e}s)$ for all $s$ with $\wt(s)=w-2$. This is equivalent to $L(\tilde{f}t)=\tilde{f}L(t)$
for all $t$ with $\wt(t)=w$.

First of all, let us remark that $\tilde{e}L(s)=0$ and $\tilde{e}s=0$ are equivalent. Indeed, we know that $h_-(L(s))=h_-(s)$
and hence $h_+(L(s))=h_+(s)$. So we may assume that $\tilde{e}s, \tilde{e}L(s)\neq 0$.
We may also assume that  $\tilde{e}L(s')=L(\tilde{e}s')$ is  proved for all $s'$ such that
 $h_-(s')=h_-(s)$ and $\tilde{e}s\succ\tilde{e}s'$.

Let $\tilde{s}$ denote the $n$-tuple with $\tilde{e}L(s)=L(\tilde{s})$. By what we have seen above, $h_-(\tilde{e}s)=h_-(\tilde{s})=h_-(s)+1$. So $\tilde{e}L(s)$ is one of the simple modules $L$ appearing
in $\head(E\Delta(s))$ with $h_-(L)=h_-(s)+1$. 

Assume that $\tilde{s}\neq \tilde{e}s$.
Let us observe that $\tilde{s}\succ \tilde{e}s$. Indeed, otherwise $\tilde{e}s\succ \tilde{s}$.
Since $h_-(\tilde{s})-1=h_-(s)\geqslant 0$, we see that $s':=\tilde{f}\tilde{s}\neq 0$ and so $\tilde{s}=\tilde{e} s'$.
Since $\tilde{e}s\succ \tilde{e}s'$ and $h_-(s')=h_-(\tilde{s})-1=h_-(s)$,
we can use the inductive assumption and get
$\tilde{e}L(s')=L(\tilde{e}s')=L(\tilde{s})$. But
$\tilde{e}L(s)=L(\tilde{s})$ hence $s=s'$ or, equivalently,
$\tilde{s}=\tilde{e}s$.

So  $\tilde{s}\succ \tilde{e}s$. But then Lemma \ref{Lem:comb11} implies $h_-(\tilde{s})>h_-(\tilde{e} s)$. So we get a contradiction which proves $\tilde{e}L(s)=L(\tilde{e}s)$.
The equality $L(\tilde{f}t)=\tilde{f}L(t)$ for all $t$ with $\wt(t)=w$ follows and we have completed the induction
step for our claim in the beginning of the subsection.

So we have $\tilde{f}L(t)=L(\tilde{f}t)$ as well as $h_-(t)=h_-(L(t))$ for all $t$. Together with standard properties
of crystals, this implies Theorem \ref{Thm:crystal_main}.

\subsection{Application to Cherednik algebras}
Let $\C$ be the field of complex numbers and $\Cat$ be the category $\mathcal{O}_p$ from Subsection \ref{SS_Ch_cycl}.
Pick a multipartition $\mu$ of $n$. Define the {\it depth} $D(\mu)$ inductively by setting $D(\mu)=0$ if $\tilde{e}_z \mu=0$
for all $z\in \mathbb{C}$ and $D(\mu)=1+\max_{z\in \mathbb{C}}(D(\tilde{e}_z\mu))$ else. We remark that $D(\mu)$
does depend on $p$ as the crystal structure on the set of multipartitions does.

Following \cite{Shan} and \cite{shanvasserot} we will interpret $D(\mu)$ representation theoretically. Namely, we can
view $L(\mu)$ as a $W$-equivariant coherent sheaf on $\h=\mathbb{C}^n$. Its support is known  to be of the form $W \h_{i,j}$,
where $\h_{i,j}$  is the space of all $n$-tuples $(x_1,\ldots,x_n)$ with $x_{i+1}=x_{i+2}=\ldots=x_{i+e}, x_{i+e+1}=x_{i+e+2}=\ldots=x_{i+2e},
\ldots, x_{i+(j-1)e+1}=\ldots=x_{i+je}, x_{i+je+1}=\ldots=x_n=0$. Here $e$ is the denominator of $\kappa$ if $\kappa$ is rational
(recall that we assume that $\kappa$ is non-integral) and $e=\infty$ if $\kappa$ is irrational (so we do not have $e$-tuples
of equal coordinates).  See \cite[Remark 3.7]{shanvasserot} for a proof. Now \cite[Proposition 3.16, Corollary 3.18]{shanvasserot}
imply that $i=D(\mu)$. In particular, if $\kappa$ is irrational we can recover the support of $L(\mu)$ completely.
This generalizes \cite[Corollary 6.9.3]{GL}.

\end{document}